\documentclass[review]{elsarticle}
\usepackage{amsfonts,amsmath,amssymb,epsfig,graphicx,hyperref,
color, esint, fancyhdr, enumerate, latexsym, lineno}

\usepackage{lineno,hyperref}
\modulolinenumbers[5]

\newtheorem{thm}{Theorem}[section]
\newproof{proof}{Proof}
\newtheorem{prop}[thm]{Proposition}
\newtheorem{lem}[thm]{Lemma}
\newtheorem{cor}[thm]{Corollary}

\newtheorem{rem}[thm]{Remark}

\newproof{pf}{Proof}

\newcommand{\mbb}{\mathbb}
\newcommand{\de}{\delta}
\newcommand{\ov}{\overline}
\newcommand{\La}{\Lambda}
\newcommand{\pa}{\partial}
\newcommand{\mf}{\mathbb}

\newcommand{\Om}{\Omega}
\newcommand{\al}{\alpha}
\newcommand{\be}{\beta}
\newcommand{\ga}{\gamma}
\newcommand{\z}{\zeta}
\newcommand{\la}{\lambda}
\newcommand{\ti}{\tilde}
\newcommand{\De}{\Delta}

\renewcommand{\Re}{\operatorname{Re}}
\renewcommand{\Im}{\operatorname{Im}}
\newcommand{\supp}{\operatorname{supp}}
\newcommand{\Ric}{\operatorname{Ric}}
\newcommand{\diag}{\operatorname{diag}}

\newcommand{\ad}{\operatorname{ad}}

\newcommand{\proofofref}{}
\newproof{zproofof}{Proof of \proofofref}
\newenvironment{proofof}[1]
 {\renewcommand{\proofofref}{#1}\zproofof}
 {\endzproofof}

\numberwithin{equation}{section}
\hypersetup{
    colorlinks,
    citecolor=blue,
    filecolor=black, 
    linkcolor=red,
    urlcolor=black
}

\begin{document}
\begin{frontmatter}
\title{Some remarks on the Kobayashi--Fuks metric on strongly pseudoconvex domains}

\author{Diganta Borah}
\address{Indian Institute of Science Education and Research, Pune, India}
\ead{dborah@iiserpune.ac.in}

\author{Debaprasanna Kar\fnref{myfootnote}}
\address{Indian Institute of Science Education and Research, Pune, India}
\ead{debaprasanna.kar@students.iiserpune.ac.in}
\fntext[myfootnote]{The second author was supported by the Council of Scientific \& Industrial Research (File no.
09/936(0221)/2019-EMR-I) and by the Ph.D. program at the Indian Institute of Science Education
and Research, Pune.}


\begin{abstract}
The Ricci curvature of the Bergman metric on a bounded domain $D\subset \mathbb{C}^n$ is strictly bounded above by $n+1$ and consequently $\log (K_D^{n+1}g_{B,D})$, where $K_D$ is the Bergman kernel for $D$ on the diagonal and $g_{B, D}$  is the Riemannian volume element of the Bergman metric on $D$, is the potential for a K\"ahler metric on $D$ known as the Kobayashi--Fuks metric. In this note we study the localization of this metric near holomorphic peak points and also show that this metric shares several properties with the Bergman metric on strongly pseudoconvex domains.
\end{abstract}


\end{frontmatter}



\section{Introduction}
For a bounded domain $D\subset\mf{C}^n$ the space
\[
A^2(D)=\left\{\text{$f: D \to \mf{C}$ holomorphic and $\|f\|^2_D:=\int_{D} \vert f\vert^2 \, dV< \infty$} \right\},
\]
where $dV$ is the Lebesgue measure on $\mf{C}^n$ is a closed subspace of $L^2(D)$ and is a reproducing kernel Hilbert space. The associated reproducing kernel denoted by $K_{D}(z,w)$ is uniquely determined by the following properties: $K_{D}(\cdot, w) \in A^2(D)$ for each $w \in D$, it is anti-symmetric, i.e., $K_{D}(z,w)=\ov{K_{D}(w,z)}$, and it reproduces $A^2(D)$:
\[
f(w)=\int_{D} f(z) \ov{K_{D}(z,w)} \, dV(z), \quad f \in A^2(D).
\]
It also follows that for any complete orthonormal basis $\{\phi_k\}$ of $A^2(D)$,
\[
K_{D}(z,w)=\sum_{k} \phi_k(z) \ov{\phi_{k}(w)},
\]
where the series converges uniformly on compact subsets of $D \times D$. The reproducing kernel $K_D(z,w)$ is called the Bergman kernel for $D$. Denote by $K_D(z)=K_D(z,z)$ its restriction to the diagonal. It is known that $\log K_D$ is a strongly plurisubharmonic function
and thus is a potential for a K\"ahler metric which is called the Bergman metric for $D$ and is given by
\[
ds^2_{B,D}=\sum_{\al,\be=1}^n g^{B,D}_{\al\ov \be} (z) \, dz_{\al}d\ov z_{\be},
\]
where
\[
g^{B,D}_{\al \ov \be}(z)=\frac{\pa^2 \log K_{D}}{\pa z_{\al} \pa \ov z_{\be}}(z).
\]
Let
\[
G_{B,D}(z)=\begin{pmatrix}g^{B,D}_{\al \ov \be}(z)\end{pmatrix}_{n \times n} \quad \text{and} \quad  g_{B,D}(z)=\det G_{B,D}(z).
\]
The components of the Ricci tensor of $ds^2_{B,D}$ are defined by
\begin{equation}\label{Ric-h-ten}
\Ric_{\al \ov \be}^{B,D}(z)=  - \frac{\pa^2 \log g_{B,D}}{\pa z_{\al} \pa \ov z_{\be}}(z),
\end{equation}
and the Ricci curvature of $ds^2_{B,D}$ is given by
\begin{equation}\label{Ric_h-curv}
\text{Ric}_{B,D}(z,u)=\frac{\sum_{\al, \be=1}^n \Ric_{\al \ov \be}^{B,D}(z) u^{\al} \ov u^{\be}}{\sum_{\al,\be=1}^n g^{B,D}_{\al\ov \be}(z)u^\al \ov u^\be}.
\end{equation}
Kobayashi \cite{Kob59} showed that the Ricci curvature of the Bergman metric on a bounded domain in $\mf C^n$ is strictly bounded above by $n+1$ and hence the matrix
\[
G_{\ti B,D}(z)=\begin{pmatrix}g^{\ti B,D}_{\al \ov \be}(z)\end{pmatrix}_{n \times n}
\]
where
\[
g^{\ti B,D}_{\al \ov \be}(z)=(n+1)g^{B,D}_{\al\ov \be}(z)-\Ric^{B,D}_{\al\ov \be}(z) = \frac{\pa^2 \log (K_D^{n+1}g_{B,D})}{\pa z_{\al} \pa \ov z_{\be}}(z),
\]
is positive definite (see also Fuks \cite{Fuks66}). Therefore,
\[
ds^2_{\ti B,D}=\sum_{\al,\be=1}^n g^{\ti B,D}_{\al \ov \be}(z)\,dz_{\al}d\ov z_{\be}
\]
is a K\"ahler metric with K\"ahler potential $\log (K_D^{n+1} g_{B,D})$. Moreover, if $F: D \to D'$ is a biholomorphism, then
\begin{equation}\label{tr-kf}
G_{\ti B, D}(z)= F'(z)^t\, G_{\ti B, D'}\big(F(z)\big) \ov F'(z),
\end{equation}
where $F'(z)$ is the Jacobian matrix of $F$ at $z$. This implies that $ds^2_{\ti B,D}$ is an invariant metric. We will call this metric the Kobayashi--Fuks metric on $D$.

The boundary asymptotics of the Bergman metric and its Ricci curvature on strongly pseudoconvex domains are known from which it turns out that the Kobayashi--Fuks metric is complete on such domains (to be seen later in Section~5). Dinew \cite{Dinew13} showed that on any bounded hyperconvex domain the Kobayashi--Fuks metric is complete, and hence in particular, by a result of Demailly \cite{Dem87}, this metric is complete on any bounded pseudoconvex domain with Lipschitz boundary. Dinew \cite{Dinew11} also observed that the Kobayashi--Fuks metric is useful in the study of Bergman representative coordinates. Invariant metrics play an important role in understanding the geometry of a domain which makes their study of natural interest in complex analysis and the purpose of this note is to show that the Kobayashi--Fuks metric shares several properties with the Bergman metric. Let us fix some notations before we state our results. We will denote by $h$ any of $B$ or $\ti B$. We write
\[
G_{h,D}(z)=\begin{pmatrix} g^{h,D}_{\al \ov \be}(z)\end{pmatrix}_{n \times n} \quad \text{and} \quad g_{h,D}(z)=\det G_{h,D}(z).
\]
The length of a vector $u$ at a point $z \in D$ in $ds^2_{h,D}$ will be denoted by $\tau_{h,D}(z,u)$, i.e.,
\[
\tau_{h,D}^2(z,u)=\sum_{\al,\be=1}^n g^{h,D}_{\al \ov \be}(z)u^{\al} \ov u^{\be}.
\]
The holomorphic sectional curvature of $ds^2_{h,D}$ is defined by
\begin{equation}\label{hsc-h}
R_{h,D}(z,u)=\frac{\sum_{\al,\be,\ga,\de=1}^n R^{h,D}_{\ov \al \be \ga \ov \de}(z)\ov u^{\al} u^{\be} u^{\ga} \ov u^{\de}}{\big(\sum_{\al,\be=1}^n g^{h,D}_{\al\ov \be}(z)u^\al \ov u^\be\big)^2},
\end{equation}
where
\begin{equation}
R^{h,D}_{\ov \al \be \ga \ov \de}(z)=-\frac{\pa^2 g^{h,D}_{\be \ov \al}}{\pa z_{\ga} \pa \ov z_{\de}}(z)+\sum_{\mu,\nu} g_{h,D}^{\nu \ov \mu}(z)\frac{\pa g^{h,D}_{\be \ov \mu}}{\pa z_{\ga}}(z)\frac{\pa g^{h,D}_{\nu \ov \al}}{\pa \ov z_{\de}}(z),
\end{equation}
$g_{h,D}^{\nu \ov \mu}(z)$ being the $(\nu,\mu)$th entry of the inverse of the matrix $G_{h,D}(z)$.  The Ricci curvature of $ds^2_{h,D}$ is defined by \eqref{Ric_h-curv} with $B$ replaced by $h$. Finally, note that in dimension one, the metric $ds^2_{h,D}$ has the form
\[
ds^2_{h,D}=g_{h,D}(z) \vert dz \vert^2, \quad \tau_{h, D}(z,u)=\sqrt{g_{h,D}(z)}\vert u\vert
\]
and both the holomorphic sectional curvature and the Ricci curvature at a point are independent of the tangent vector $u$ and are simply the Gaussian curvature
\[
R_{h,D}(z)=-\frac{1}{g_{h,D}(z)}\frac{\pa^2 \log g_{h,D}}{\pa z \pa \ov z}(z).
\]

Our first result is on the localization of the Kobayashi--Fuks metric near holomorphic peak points.
\begin{thm}\label{loc}
Let $D \subset \mf{C}^n$ be a bounded pseudoconvex domain with a holomorphic peak point $p^0 \in \pa D$. If $U$ is a sufficiently small neighborhood of $p^0$, then
\begin{itemize}
\item [{\em(I)}] $\lim_{z \to p^0} \frac{\tau_{\ti B,D}(z,u)}{\tau_{\ti B,U \cap D}(z,u)}=1$ uniformly in unit vectors $u \in \mf{C}^n$.
\item [{\em(II)}] $\lim_{z \to p^0} \frac{g_{\ti B, D}(z)}{g_{\ti B, U\cap D}(z)}=1$.
\item [{\em(III)}] If $n=1$, then $\lim_{z \to p^0} \frac{2-R_{\ti B, D}(z)}{2-R_{\ti B, U\cap D}(z)}=1$.
\end{itemize}
\end{thm}

A crucial step in the proof of this theorem is to obtain Bergman--Fuks type results for the Kobayashi--Fuks metric and its related invariants, i.e., to express them in terms of certain maximal domain functions. For the holomorphic sectional curvature of the Kobayashi--Fuks metric, we derive such a result only in dimension one, though we believe that in higher dimensions also, an analog of this and hence of (III) above should hold.

Next we investigate the boundary behavior of the Kobayashi--Fuks metric on strongly pseudoconvex domains in $\mf{C}^n$. We will denote by $\de_D(z)$ the Euclidean distance from the point $z\in D$ to the boundary $\pa D$. For $z$ close to the boundary $\pa D$, let $\pi(z)\in \pa D$ be the nearest point to $z$, i.e., $\de_D(z)=|z-\pi(z)|$, and for a tangent vector $u\in \mf C^n$ based at $z$, let $u=u_H(z)+u_N(z)$ be the decomposition along the tangential and normal directions respectively at $\pi(z)$.

\begin{thm}\label{bdy-n}
Let $D \subset \mf{C}^n$ be a $C^2$-smoothly bounded strongly pseudoconvex domain and $p^0 \in \pa D$. Then there are holomorphic coordinates $z$ near $p^0$ in which
\begin{itemize}
\item [\em{(I)}] $\de_{D}(z)\, \tau_{\ti B,D}(z,u)\to \frac{1}{2}\sqrt{(n+1)(n+2)}\,\vert u_N(p^0)\vert$\,,
\item [\em{(II)}] $\sqrt{\de_{D}(z)}\, \tau_{\ti B, D}\big(z,u_{H}(z)\big)\to \sqrt{\frac{1}{2}(n+1)(n+2)\mathcal{L}_{\pa D}\big(p^0,u_H(p^0)\big)}$\,, 
\item [\em{(III)}] $\de_D(z)^{n+1} g_{\ti B,D}(z) \to \frac{(n+1)^n(n+2)^n}{2^{n+1}}$,
\item [\em{(IV)}] If $n=1$, then $R_{\ti B, D}(z) \to -\frac{1}{3}$,
\end{itemize}
as $z \to p^0$. Here, $\mathcal{L}_{\pa D}$ is the Levi form of $\pa D$ with respect to some defining function for $D$.
\end{thm}

Note that Theorem~\ref{bdy-n}~(I), (II) are analogs of Graham's result \cite{Gr} for the Kobayashi and Carath\'{e}odory metrics. Also, Theorem~\ref{bdy-n}~(IV) combined with Theorem~1.17 of \cite{GK} immediately yields

\begin{cor}
Let $D,D' \subset \mbb{C}$ be $C^2$-smoothly bounded domains equipped with the metrics $ds^2_{\ti B,D}$ and $ds^2_{\ti B,D'}$ respectively. Then every isometry $f: (D, ds^2_{\ti B,D}) \to (D', ds^2_{\ti B,D'})$ is either holomorphic or conjugate holomorphic.
\end{cor}
\medskip

Our final result is motivated by a theorem of Herbort  \cite[Theorem~1.2]{Herbort1983}
on the existence of closed geodesics for the Bergman metric on strongly pseudoconvex domains.

\begin{thm}\label{geodesic}
Let $D \subset \mf{C}^n$ be a smoothly bounded strongly pseudoconvex domain which is not simply connected. Then every nontrivial homotopy class in $\pi_1(D)$ contains a closed geodesic for $ds^2_{\ti B, D}$.
\end{thm}

\medskip

{\it Acknowledgements}: The authors thank Kaushal Verma for his support and encouragement.

\section{Some examples}
\begin{prop}\label{ball}
For the unit ball $\mf{B}^n \subset \mf{C}^n$,
\[
ds^2_{\ti B, \mathbb{B}^n}=(n+2)ds^2_{B, \mathbb{B}^n} = (n+1)(n+2)\sum_{\al,\be=1}^n \left(\frac{\de_{\al\ov\be}}{1-\vert z \vert^2}+\frac{\ov z_{\al}z_{\be}}{(1-\vert z \vert^2)^2}\right) dz_{\al} d\ov z_{\be}.
\]
\end{prop}

\begin{proof}
Recall that for the unit ball $\mf{B}^n \subset \mf{C}^n$,
\[
K_{\mf B^n}(z)=\frac{n!}{\pi^n}\frac{1}{(1-\vert z \vert^2)^{n+1}},
\]
and so
\[
g_{\al\ov \be}^{B,\mf{B}^n}(z)=(n+1)\frac{\pa^2}{\pa z_{\al} \pa \ov z_{\be}} \log\frac{1}{1-\vert z\vert^2}=(n+1)\left(\frac{\de_{\al\ov\be}}{1-\vert z \vert^2}+\frac{\ov z_{\al}z_{\be}}{(1-\vert z \vert^2)^2}\right).
\]
Denoting the matrix $\ov z z^t$ by $A_z$ and using the fact that its characteristic polynomial is $\det(\la \mf{I}-A_z)=\la^n-\vert z \vert^2 \la^{n-1}$, we obtain
\[
g_{B,\mf{B}^n}(z)=\frac{(n+1)^n}{(1-\vert z \vert^2)^{n+1}},
\]
and hence
\[
\Ric^{B,\mf{B}^n}_{\al\ov \be}(z) = -(n+1)\frac{\pa^2}{\pa z_{\al}\pa \ov z_{\be}} \log\frac{1}{1-\vert z \vert^2}=-g^{B,\mf{B}^n}_{\al \ov \be}(z).
\]
It follows that
\[
g^{\ti B, \mf{B}^n}_{\al \ov \be}(z)=(n+2)g^{B, \mf{B}^n}_{\al\ov \be}(z),
\]
which completes the proof of the proposition.\qed
\end{proof}

\begin{prop}\label{Polydisc}
For the unit polydisc $\De^n \subset \mf{C}^n$,
\[
ds^2_{\ti B, \De^n}=(n+2)ds^2_{B, \De^n}=2(n+2)\sum_{\al=1}^n\dfrac{1}{(1-|z_{\al}|^2)^2}dz_{\al} d\ov z_{\al}.
\]
\end{prop}
\begin{proof}
For the unit polydisc $\De^n \subset \mf{C}^n$, recall that by the product formula for the Bergman kernel,
\[
K_{\De^n}(z)=\prod_{j=1}^n K_{\De}(z_j)=\frac{1}{\pi^n}\prod_{j=1}^n \frac{1}{(1-\vert z_j\vert^2)^2} ,
\]
and therefore
\[
g_{\al\ov \be}^{B,\De^n}(z)=2\,\frac{\pa^2}{\pa z_{\al} \pa \ov z_{\be}}\sum_{j=1}^n \log\frac{1}{1-\vert z_j\vert^2}=\dfrac{2\de_{\al \be}}{(1-|z_{\al}|^2)^2}.
\]
Thus
\[
g_{B,\De^n}(z)=2^n\prod_{j=1}^n\dfrac{1}{(1-|z_j|^2)^2} ,
\]
and hence
\[
\Ric^{B,\De^n}_{\al\ov \be}(z) =-2\,\frac{\pa^2}{\pa z_{\al}\pa \ov z_{\be}}\sum_{j=1}^n \log\frac{1}{1-\vert z_j \vert^2}=-g^{B,\De^n}_{\al \ov \be}(z).
\]
It follows that
\[
g^{\ti B, \De^n}_{\al \ov \be}(z)=(n+2)g^{B, \De^n}_{\al\ov \be}(z)=\dfrac{2(n+2)}{(1-|z_{\al}|^2)^2}\de_{\al \be},
\]
and the proof of the proposition is complete.\qed
\end{proof}

In general, if $D$ is a bounded domain with a transitive group of holomorphic automorphisms, the Bergman metric is K\"ahler-Einstein, and so $ds^2_{\ti B, D}$ is a constant multiple of $ds^2_{B, D}$.

\section{Localization}
Our goal in this section is to prove Theorem \ref{loc}. The localization of the Bergman kernel, the Bergman metric, and various invariants related to them near holomorphic peak points on pseudoconvex domains are well known, see for example \cite{KY, Krantz-Yu} for bounded domains, and \cite{Nik2002, Nik2015, Nik2020} for unbounded domains. As in the case of the Bergman metric, the main idea is to express the Kobayashi--Fuks metric and the other invariants in terms of certain maximal domain functions. We will show that $\tau_{\ti B,D}$ can be expressed in terms of a maximal domain function introduced by Krantz and Yu in \cite{Krantz-Yu}. However, for the Gaussian curvature of $ds^2_{\ti B,D}$, we will require some new maximal domain functions. We begin by recalling the definition of the domain function of Krantz and Yu: For a bounded domain $D \subset \mathbb{C}^n$, $z_0 \in D$, and a nonzero vector $u \in \mathbb{C}^n$, let
\begin{equation} \label{I}
I_{D}(z_0,u)=\sup \left\{u^tf''(z_0)\ov G_{B, D}^{-1}(z_0)\ov{f''}(z_0)\ov u: \|f\|_{D}=1, f(z_0)=f'(z_0)=0\right\}.
\end{equation}
Here $f''(z_0)$ is the symmetric matrix
\[
f''(z_0)=\begin{pmatrix}\frac{\pa^2 f}{\pa z_i \pa z_j}(z_0)\end{pmatrix}_{n \times n}.
\]
It was shown in Proposition~2.1 (ii) of \cite{Krantz-Yu} that
\begin{equation}\label{Ric-Berg-I}
\Ric_{B, D}(z_0,u)=(n+1)-\frac{1}{\tau^2_{B, D}(z_0,u)K_{D}(z_0)}I_{D}(z_0,u).
\end{equation}
Also, from the definition of the Kobayashi--Fuks metric, note that
\begin{equation}\label{B-KF}
\tau_{\ti B, D}(z_0,u)=\tau_{B, D}(z_0,u)\sqrt{n+1-\Ric_{B, D}(z_0,u)}.
\end{equation}
Combining \eqref{Ric-Berg-I} and \eqref{B-KF} we obtain
\begin{prop}\label{F_metric}
Let $D$ be a bounded domain in $\mf{C}^n$, $z_0\in D$, and $u\in \mbb{C}^n$. Then we have
\begin{align*}
\tau^2_{\ti B, D}(z_0,u)= \frac{I_{D}(z_0,u)}{K_{D}(z_0)}.
\end{align*}
\end{prop}
From this proposition, we immediately obtain the localization of $\tau_{\ti B,D}$.

\begin{proofof}{Theorem~\ref{loc}~(i)}
It was shown in \cite{KY} that
\begin{equation}\label{K-loc-KY}
\lim_{z \to p^0} \frac{K_D(z)}{K_{U\cap D}(z)}=1
\end{equation}
and in \cite{Krantz-Yu} that
\begin{equation}\label{I-loc-KrY}
\lim_{z \to p^0} \frac{I_D(z,u)}{I_{U\cap D}(z,u)}=1
\end{equation}
uniformly in unit vectors $u$, and hence (I) follows from Proposition~\ref{F_metric}. \hfill $\square$
\end{proofof}

For the localization of $g_{\ti B,D}$, we will need the following lemma. The notation $\mathbb{I}_n$ or simply $\mathbb{I}$, which we have already used in the previous section, stands for the $n\times n$ identity matrix.

\begin{lem}\label{S-D}
Let $D_1, D_2$ be two bounded domains in $\mbb{C}^n$ such that $D_2 \subset D_1$. For any $z_0\in D_2$, there exist a nonsingular matrix $Q$ and positive real numbers $d_1,\ldots,d_n$ such that
\[
Q^tG_{\ti B, D_1}(z_0)\ov Q=\diag\{d_1,\ldots,d_n\} \quad\text{and}\quad Q^t G_{\ti B, D_2}(z_0) \ov Q=\mathbb{I}.
\]
\end{lem}

\begin{proof}
Note that as $G_{\ti B, D_2}(z_0)$ is a positive definite Hermitian matrix one can find an invertible matrix $A$ such that
\[
A^tG_{\ti B, D_2}(z_0) \ov A=\mathbb{I}.
\]
By the transformation rule \eqref{tr-kf} applied to $A: A^{-1}D_1 \to D_1$ and $A: A^{-1}D_2 \to D_2$,
\[
G_{\ti B, A^{-1}D_1}(A^{-1}z_0)=A^tG_{\ti B, D_1}(z_0) \ov A\quad\text{and}\quad G_{\ti B, A^{-1}D_2}(A^{-1}z_0)=A^t G_{\ti B, D_2}(z_0)\ov A=\mathbb{I}.
\]
From the first identity above, $A^tG_{\ti B, D_1}(z_0)\ov A$ is a positive definite Hermitian matrix and hence there exists a unitary matrix $B$ such that
\[
B^t(A^t G_{\ti B, D_1}(z_0) \ov A)\ov B=\diag\{d_1,\ldots,d_n\}\quad \text{for some}\quad d_1,\ldots,d_n>0.
\]
Now letting $Q=AB$ the lemma follows.\qed
\end{proof}

\begin{proofof}{Theorem~\ref{loc} (II)}
By Lemma \ref{S-D}, there exist an invertible matrix $Q(z)$ and positive real numbers $d_1(z), \ldots, d_n(z)$ such that 
\begin{equation*}
Q^t(z)G_{\ti B, D}(z) \ov Q(z)=\diag \{d_1(z),\ldots,d_n(z)\} \quad \text{and} \quad Q^t(z)G_{\ti B, U\cap D}(z)\ov Q(z)=\mathbb{I}.
\end{equation*}
Taking determinant on both sides of these equations yields
\[
\frac{g_{\ti B, D}(z)}{g_{\ti B, U\cap D}(z)}=\prod\limits_{j=1}^n d_j(z).
\]
Also, by Proposition~\ref{F_metric},
\begin{equation*}
\tau^2_{\ti B, D}(z,u)=\frac{K_{U\cap D}(z)}{K_{D}(z)}\frac{I_{D}(z,u)}{I_{U\cap D}(z,u)}\,\tau^2_{\ti B, U\cap D}(z,u).
\end{equation*}
Putting $u=Q(z)e_j$ in the above equation, we get 
\[
d_j(z)=\frac{K_{U\cap D}(z)}{K_{D}(z)}\frac{I_{D}\big(z,Q(z)e_j\big)}{I_{U\cap D}\big(z,Q(z)e_j\big)}\quad \text{for}\quad j=1,\ldots,n.
\]
Therefore,
\begin{multline*}
\frac{g_{\ti B, D}(z)}{g_{\ti B, U\cap D}(z)}=\left(\frac{K_{U\cap D}(z)}{K_{D}(z)}\right)^n\,\prod\limits_{j=1}^n \frac{I_{D}\big(z,Q(z)e_j\big)}{I_{U\cap D}\big(z,Q(z)e_j\big)}
=\left(\frac{K_{U\cap D}(z)}{K_{D}(z)}\right)^n\,\prod\limits_{j=1}^n \frac{I_{D}\big(z,v_j(z)\big)}{I_{U\cap D}\big(z,v_j(z)\big)}
\end{multline*}
where $v_j(z)=Q(z)e_j/\|Q(z)e_j\|$. Now (II) follows immediately from \eqref{K-loc-KY} and \eqref{I-loc-KrY}. \qed
\end{proofof}

In general, the Kobayashi--Fuks metric and its associated objects do not satisfy monotonicity property. Nevertheless, we show that they can be compared. Recall that the Bergman canonical invariant on $D$ is the function defined by
\[
J_D(z)= \frac{\det G_{B,D}(z)}{K_D(z)} = \frac{g_{B,D}(z)}{K_D(z)}.
\]
From the transformation rule for the Bergman kernel it is evident that $J_D$ is a biholomorphic invariant. 

\begin{prop}\label{Monotone}
Let $D_1, D_2$ be two bounded domains in $\mbb{C}^n$ such that $D_2 \subset D_1$. For any $z_0\in D_2$ and $u\in \mbb{C}^n$, we have
\begin{itemize}
\item[\em{(a)}]\,$\tau^2_{\ti B, D_1} (z_0,u)\leq \left(\frac{K_{D_2}(z_0)}{K_{D_1}(z_0)}\right)^{n+1}\left(\frac{J_{D_2}(z_0)}{J_{D_1}(z_0)}\right)\, \tau^2_{\ti B, D_2} (z_0,u),$
\item[\em{(b)}]\,$u^t \ov G_{\ti B, D_1}^{-1}(z_0)\ov u \geq \left(\frac{K_{D_1}(z_0)}{K_{D_2}(z_0)}\right)^{n+1} \left(\frac{J_{D_1}(z_0)}{J_{D_2}(z_0)}\right) u^t \ov G_{\ti B, D_2}^{-1}(z_0) \ov u.$
\end{itemize}

\end{prop}

\begin{proof}
Fix $z_0 \in D_2$ and $u \in \mathbb{C}^n$. For simplicity of notations, we will write $K_i$ for $K_{D_i}(z_0)$, $J_i$ for $J_{D_i}(z_0)$, $G_i$ for $G_{B, D_i}(z_0)$, and $\ti G_i$ for $G_{\ti B, D_i}(z_0)$ for $i=1,2$.

(a) In view of Proposition~\ref{F_metric}, it is enough to prove that
\begin{equation}\label{I_1-I_2}
I_{D_1}(z_0,u) \leq \left(\frac{K_2}{K_1}\right)^{n} \frac{J_2}{J_1} \,I_{D_2}(z_0,u).
\end{equation}
From the proof of Proposition~2.2 in \cite {Krantz-Yu} (see page 236) there exists a nonsingular matrix $P$ (depending on $z_0$) such that for every $v \in \mbb{C}^n$,
\begin{equation}\label{A1}
v^t {P}^{-1}\,\ov{\ad} G_1 (P^*)^{-1}\ov v \leq \left(\frac{K_2}{K_1}\right)^{n-1} v^t P^{-1}\,\ov{\ad}G_2 (P^*)^{-1} \ov v,
\end{equation}
where by $\ad A$ and $A^*$, we mean the adjugate and the conjugate transpose of the matrix $A$ respectively. Now consider $f \in A^2(D_1)$ such that $\|f\|_{D_1}=1, f(z_0)=f'(z_0)=0$.
We can write $f''(z_0)=(P^t)^{-1}A$ for some matrix $A$. Putting $v=Au$ in (\ref{A1}) and using the fact that $f''(z_0)$ is symmetric, we get
\begin{equation}\label{A2}
u^t f''(z_0)\,\ov{\ad} G_1 \ov{f''}(z_0)\ov u \leq \left(\frac{K_2}{K_1}\right)^{n-1}\, u^t f''(z_0)\ov{\ad}G_2 \ov{f''}(z_0) \ov u.
\end{equation}
Define $g: D_2 \to \mf{C}$ by
\[
 g(z)=\frac{f(z)}{\|f\|_{D_2}}.
 \]
Then $g\in A^2(D_2)$, $\|g\|_{D_2}=1$, $g(z_0)=g'(z_0)=0$. Since $f''(z_0)=\|f\|_{D_2} g''(z_0)$, $\|f\|_{D_2}\leq 1$, and $\text{ad}\,G_i=(\text{det}\,G_i)G_i^{-1}$, we have from (\ref{A2})
\begin{multline}\label{A3}
u^t f''(z_0)\,\ov{G}_1^{-1} \ov{f''}(z_0) \ov u \leq \left(\frac{K_2}{K_1}\right)^{n-1}\frac{\text{det}\,G_2}{\text{det}\,G_1}\big(u^t g''(z_0)\,\ov{G}_2^{-1} \ov{g''}(z_0) \ov u\big) \\
\leq \left(\frac{K_2}{K_1}\right)^{n}\frac{J_2}{J_1}I_{D_2}(z_0,u).
\end{multline}
Taking supremum over $f$ in \eqref{A3} and using Proposition~\ref{F_metric}, we obtain \eqref{I_1-I_2} and hence (a) is proved.

(b) Let $Q$ be as in Lemma~\ref{S-D}. Let $e_j$ denote the j-th standard basis vector in $\mbb{C}^n$, i.e., $e_j=(0,\ldots,0,1,0,\ldots,0)$. Taking $u=Qe_j$ in (a), we get
\begin{equation}\label{A5} 
d_j \leq \left(\frac{K_2}{K_1}\right)^{n+1}\frac{J_2}{J_1}\quad \text{for}\quad j=1,\ldots,n.
\end{equation}
From Lemma \ref{S-D}, it follows that
\begin{equation*}
Q^{-1}\big(\ov{\ti{G_1}}\big)^{-1}(Q^{*})^{-1}=\text{diag}\left\{\frac{1}{d_1},\ldots,\frac{1}{d_n}\right\}\quad \text{and} \quad Q^{-1} \big(\ov{\ti{G_2}}\big)^{-1}(Q^{*})^{-1}=\mathbb{I}.
\end{equation*}
Hence for any $v\in \mbb{C}^n$, using the inequality (\ref{A5}) we get
\begin{multline*}
v^tQ^{-1}\big(\ov{\ti{G_1}}\big)^{-1}(Q^{*})^{-1} \ov v = v^t \diag \left\{\frac{1}{d_1},\ldots,\frac{1}{d_n}\right\}\ov v \geq \left(\frac{K_1}{K_2}\right)^{n+1}\frac{J_1}{J_2} (v^t \mathbb{I}\ov v)\\
=\left(\frac{K_1}{K_2}\right)^{n+1}\frac{J_1}{J_2} \big( v^t Q^{-1} \big(\ov{\ti{G_2}}\big)^{-1}(Q^{*})^{-1} \ov v\big).
\end{multline*}
Putting $u=(Q^t)^{-1} v$ in the above inequality, we get (b).\qed
\end{proof}

We now introduce two maximal domain functions on planar domains for the purpose of localizing the Gaussian curvature of the Kobayashi--Fuks metric. For a bounded domain $D \subset \mathbb{C}$, let
\begin{align*}
I'_{D}(z_0) &=\sup\{g^{-1}_{\ti B, D}(z_0)|f'(z_0)|^2: f \in A^2(D), \|f\|_{D}=1, f(z_0)=0\},\\
I''_{D}(z_0) &=\sup\{g^{-3}_{\ti B, D}(z_0)|f'''(z_0)|^2: f \in A^2(D), \|f\|_{D}=1, f(z_0)=f'(z_0)=f''(z_0)=0\}.
\end{align*}
Note that, as $D$ is bounded, the functions $I'_{D}$ and $I''_{D}$ are well-defined and strictly positive. It is also evident that the supremums are achieved. Moreover, under biholomorphisms they transform by the same rule as that of the Bergman kernel which we establish in the following proposition:
\begin{prop}\label{tr-rule-I'-I''} Let $F: D_1 \to D_2$ be a biholomorphism between two bounded domains in $\mf{C}$. Then
\begin{equation*}
I'_{D_1}(z_0)=I'_{D_2}\big(F(z_0)\big) \big| F'(z_0) \big|^2\quad \text{and}\quad I''_{D_1}(z_0)=I''_{D_2}\big(F(z_0)\big)\big|F'(z_0)\big|^2.
\end{equation*}
\end{prop}

\begin{proof}
We will prove the transformation rule only for $I''_{D}$, as the case of $I'_{D}$ is even simpler and follows from similar arguments. Suppose $g \in A^2(D_2)$ is a maximizer for $I''_{D_2}\big(F(z_0)\big)$. Now set
\[
f(z)=g\big(F(z)\big)F'(z).
\]
It is straightforward to check that $\|f\|_{D_1}=\|g\|_{D_2}=1$, $f(z_0)=f'(z_0)=f''(z_0)=0$, and
\[
f'''(z_0)=g'''\big(F(z_0)\big)\big(F'(z_0)\big)^4.
\]
Therefore,
\begin{equation}\label{identity-1}
g^{-3}_{\ti B, D_1}(z_0)\big|f'''(z_0)\big|^2=g^{-3}_{\ti B, D_1}(z_0)  \big|g'''\big(F(z_0)\big)\big|^2 \big|F'(z_0)\big|^8.
\end{equation}
Note that from the transformation rule for the Kobayashi--Fuks metric, we have
\[
g^{-1}_{\ti B, D_1}(z_0)|F'(z_0)|^2=g^{-1}_{\ti B, D_2}\big(F(z_0)\big).
\]
Applying this on the right hand side of (\ref{identity-1}), we get
\[
g^{-3}_{\ti B, D_1}(z_0)|f'''(z_0)|^2=g^{-3}_{\ti B, D_2}\big(F(z_0)\big)\big|g'''\big(F(z_0)\big)\big|^2 \big|F'(z_0)\big|^2.
\]
As $f$ is a candidate for $I''_{D_1}(z_0)$ and $g$ is a maximizer for $I''_{D_2}(F(z_0))$, we obtain
\[
I''_{D_1}(z_0) \geq I''_{D_2}\big(F(z_0)\big)\,|F'(z_0)|^2.
\]
Similar arguments when applied to the map $F^{-1}:D_2\to D_1$ gives the reverse inequality and hence it is an equality.\qed
\end{proof}

The main ingredient for the localization of the Gaussian curvature of the Kobayashi--Fuks metric is the following Bergman--Fuks type result:
\begin{prop}\label{Sec curv n=1}
Let $D\subset \mathbb{C}$ be a bounded domain and $z_0 \in D$. Then the Gaussian curvature of the Kobayashi--Fuks metric on $D$ satisfies
\begin{equation}\label{R-I}
R_{\ti B, D}(z_0)=2-\dfrac{I'_{D}(z_0)}{K_{D}(z_0)}-\dfrac{I''_{D}(z_0)}{I'_{D}(z_0)}\,.
\end{equation}
\end{prop}

Observe that both the sides of \eqref{R-I} are invariant under biholomorphisms and we will establish their equality by computing them in terms of a suitable orthonormal basis of $A^2(D)$ in some special coordinates. To this end, we fix $z_0 \in D$ and consider the closed subspaces of $A^2(D)$ given by
\begin{align*}
A_1(z_0)&=\big\{f\in A^2(D):f(z_0)=0\big\},\\
A_2(z_0)&=\big\{f\in A^2(D):f(z_0)=f'(z_0)=0\big\},\\
A_3(z_0)&=\big\{f\in A^2(D):f(z_0)=f'(z_0)=f''(z_0)=0\big\}.
\end{align*}
Observe that the orthogonal complement of $A_1(z_0)$ in $A^2(D)$ has dimension one and let $h_0$ be a unit vector in this orthogonal complement. It is easy to see that the orthogonal complement of $A_2(z_0)$ in $A_1(z_0)$ has dimension at most one. Since $D$ is bounded, this space contains the function $z-z_0$, and hence its dimension is exactly one. Similarly, the orthogonal complement of $A_3(z_0)$ in $A_2(z_0)$ has dimension one. Let $\{\phi,\psi,h_1,h_2,\ldots\}$ be an orthonormal basis for $A_1(z_0)$ such that $\phi$ is a unit vector in $A_1(z_0)\setminus A_2(z_0)$, $\psi$ is a unit vector in $A_2(z_0)\setminus A_3(z_0)$, and $\{h_1,\ldots, h_j,\ldots\}$ is an orthonormal basis for $A_3(z_0)$. Note that
\begin{equation}\label{Kernel}
K_{D}(z)=\big|h_0(z)\big|^2+\big|\phi(z)\big|^2+\big|\psi(z)\big|^2+\sum\limits_{j=1}^{\infty}\big|h_j(z)\big|^2,\quad z\in D.
\end{equation}
Hence $K_{D}(z_0)=\big|h_0(z_0)\big|^2$, which in particular implies $h_0(z_0)\neq 0$.

\begin{lem}\label{properties_I}
In normal coordinates for the Kobayashi--Fuks metric at $z_0$,
\begin{itemize}
\item [\em{(a)}] $I'_{D}(z_0)=\left|\phi'(z_0)\right|^2$, and
\item [\em{(b)}] $I''_{D}(z_0)=\sum_{j=1}^{\infty}\left \vert h_j'''(z_0)\right\vert^2$.
\end{itemize}
\end{lem}

\begin{proof}
(a) Observe that in normal coordinates at $z_0$, $I'_{D}$ is reduced to
\[
I'_{D}(z_0)=\sup \left\{|f'(z_0)|^2: f\in A_1(z_0), \|f\|_{D}=1\right\}.
\]
Since $\phi$ is a candidate for $I'_{D}(z_0)$,
\[
I'_{D}(z_0)\geq \left |\phi'(z_0)\right|^2.
\]
To see the reverse inequality, consider any $f\in A_1(z_0)$ with $\|f\|_{D}=1$. Since $f$ can be represented as
\[
f(z)=\langle f,\phi \rangle \,\phi(z)+\langle f,\psi \rangle \,\psi(z)+\sum_{j=1}^{\infty}\langle f,h_j \rangle \,h_j(z),
\]
using $\|f\|_{D}=1$, we have
\[
|f'(z_0)|^2=\big|\langle f,\phi \rangle\big|^2 \big|\phi'(z_0)\big|^2\leq \big|\phi'(z_0)\big|^2.
\]
Now taking supremum in the left hand side of the above inequality, we get
\[
I'_{D}(z_0) \leq \big|\phi'(z_0)\big|^2.
\]

(b) Clearly the right hand side is finite thanks to Cauchy estimates. Observe that in normal coordinates at $z_0$,
\[
I''_{D}(z_0)=\sup\left\{|f'''(z_0)|^2: f \in A_3(z_0), \|f\|_{D}=1\right\}.
\]
Note that if $f(z)=\sum_{j=1}^{\infty}a_jh_j(z)$ is an arbitrary member of $A_3(z_0)$, then
\[
\big|f'''(z_0)\big|^2=\left|\sum\limits_{j=1}^{\infty}a_j h_j'''(z_0) \right|^2=\left|\sum\limits_{j=1}^{\infty}a_jH_j\right|^2,
\]
where $H_j=h_j'''(z_0)$. Moreover, as $A_3(z_0)$ is linearly isometric to $\ell_2$ via the orthonormal basis $\{h_j\}$, we also have $\|f\|_D=\|a\|_{\ell_2}$, where $a=\{a_j\}_{j\geq 1}$.
Hence, we arrive at
\begin{equation}\label{I''-l_2}
I''_{D}(z_0)=\sup\left\{\bigg|\sum\limits_{j=1}^{\infty}a_jH_j\bigg|^2: a=\{a_j\}_{j\geq 1}\in \ell_2, \|a\|=1\right\}.
\end{equation}
Now let $H=\{H_j\}_{j=1}^{\infty}$ and define $L_H : \ell_2 \to \mf{C}$ by
\[
L_H(a)=\sum_{j=1}^{\infty} a_jH_j.
\]
Then $L_H$ is a bounded linear operator on $\ell_2$ and denoting its operator norm by $\|L_H\|$, observe that
\begin{equation}\label{L_H-I''}
\|L_H\|^2=\sup\limits_{\|a\|=1}\Big|L_H\big(\{a_j\}\big)\Big|^2=\sup\limits_{\|a\|=1}\bigg|\sum\limits_{j=1}^{\infty}a_jH_j\bigg|^2=I''_{D}(z_0)
\end{equation}
by \eqref{I''-l_2}. Also, from the canonical isometry of $\ell_2'$ with $\ell_2$, we have
\begin{equation}\label{L_H-H}
\|L_H\|^2 = \|H\|_{\ell_2}^2=\sum\limits_{j=1}^{\infty}|H_j|^2=\sum\limits_{j=1}^{\infty}\left|h_j'''(z_0)\right|^2.
\end{equation}
From \eqref{L_H-I''} and \eqref{L_H-H}, the lemma follows immediately.\qed
\end{proof}

\begin{proofof}{Proposition~\ref{Sec curv n=1}}
We work in normal coordinates for the Kobayashi--Fuks metric at $z_0$. Without loss of generality, we will denote the new coordinates by $z$, same as the previous ones. Note that in normal coordinates at $z_0$, we have
\begin{equation}\label{n-cor-g}
g_{\ti B,D}(z_0)=1, \quad \frac{\pa g_{\ti B,D}}{\pa z}(z_0)=0, \quad \text{and} \quad R_{\ti B,D}(z_0)= -\frac{\pa^2 g_{\ti B,D}}{\pa z \pa \ov z}(z_0).
\end{equation}
Next, we express the above equations in terms of the basis expansion of $K_{D}$. Recall that the K\"ahler potential of the Kobayashi--Fuks metric in dimension $1$ is $\log A(z)$ where
\begin{equation}\label{pot-A}
A=K_D^2g_{B,D}=K_D^2 \frac{\pa^2 \log K_D}{\pa z \pa \ov z}=K_D\tfrac{\pa^2 K_D}{\pa z \pa \ov z}-\tfrac{\pa K_D}{\pa z}\tfrac{\pa K_D}{\pa \ov z}.
\end{equation}
Thus
\begin{align}\label{Fuks_tensor}
g_{\ti B, D}=\frac{\pa^2 \log A}{\pa z \pa \ov z}=\dfrac{1}{A}\dfrac{\pa^2 A}{\pa z \pa \ov z}-\dfrac{1}{A^2}\dfrac{\pa A}{\pa z}\dfrac{\pa A}{\pa \ov z}.
\end{align}
Using the expansion of the Bergman kernel as in (\ref{Kernel}), we get from \eqref{pot-A}
\begin{equation}\label{der-A}
\begin{aligned}
A(z_0) & =\big|h_0(z_0)\big|^2\,\left|\phi'(z_0)\right|^2,\\
\frac{\pa A}{\pa z}(z_0) & = \big|h_0(z_0)\big|^2\,\ov{\phi'(z_0)}\,\phi''(z_0),\\
\dfrac{\pa^2 A}{\pa z\pa\ov z}(z_0)&=\big|h_0(z_0)\big|^2\left(\left|\phi''(z_0)\right|^2+\left|\psi''(z_0)\right|^2\right),\\
\dfrac{\pa^2 A}{\pa z^2}(z_0)& = \big|h_0(z_0)\big|^2 \ov{\phi'}(z_0)\phi'''(z_0)+\ov{h_0(z_0)}\,h_0'(z_0)\ov{\phi'}(z_0)\phi''(z_0)\\ &\hspace{4cm}-\ov{h_0(z_0)}\,h_0''(z_0)\left\vert\phi'(z_0)\right\vert^2, \quad \text{and}\\
\dfrac{\pa^3 A}{\pa z^2 \pa\ov z}(z_0)&=\big|h_0(z_0)\big|^2\bigg(\ov{\phi''}(z_0)\phi'''(z_0)+\ov{\psi''}(z_0)\psi'''(z_0)\bigg)-\ov{h_0(z_0)}\,h_0''(z_0)\phi'(z_0)\ov{\phi''}(z_0)\\
&\hspace{3cm}+\ov{h_0(z_0)}\,h_0'(z_0)\left(\left\vert\phi''(z_0)\right\vert^2+\left\vert\psi''(z_0)\right\vert^2\right).
\end{aligned}
\end{equation}
Also, differentiating (\ref{Fuks_tensor}) one immediately obtains
\begin{equation}\label{relation_3}
\dfrac{\pa g_{\ti B,D}}{\pa z}=\dfrac{1}{A}\dfrac{\pa^3 A}{\pa z^2 \pa\ov z}-\dfrac{2}{A^2}\dfrac{\pa A}{\pa z}\dfrac{\pa^2 A}{\pa z \pa\ov z}-\dfrac{1}{A^2}\dfrac{\pa A}{\pa \ov z}\dfrac{\pa^2 A}{\pa z^2}+\dfrac{2}{A^3}\dfrac{\pa A}{\pa \ov z}\bigg(\dfrac{\pa A}{\pa z}\bigg)^2,
\end{equation}
and
\begin{multline}\label{2nd-der-g}
-\dfrac{\pa^2 g_{\ti B,D}}{\pa z\pa \ov{z}}(z) =-\dfrac{1}{A}\dfrac{\pa^4 A}{\pa z^2 \pa \ov z^2}+\dfrac{2}{A^2}\dfrac{\pa A}{\pa\ov z}\dfrac{\pa^3 A}{\pa z^2 \pa \ov z}+\dfrac{2}{A^2}\dfrac{\pa A}{\pa z}\dfrac{\pa^3 A}{\pa z \pa \ov z^2}\\
+\dfrac{2}{A^2}\left(\dfrac{\pa^2 A}{\pa z \pa \ov z}\right)^2+\dfrac{1}{A^2}\dfrac{\pa^2 A}{\pa z^2}\dfrac{\pa^2 A}{\pa \ov z^2}
+\dfrac{6}{A^4}\left(\dfrac{\pa A}{\pa z}\right)^2\left(\dfrac{\pa A}{\pa \ov z}\right)^2\\
-\dfrac{2}{A^3}\left(\dfrac{\pa A}{\pa \ov z}\right)^2\dfrac{\pa^2 A}{\pa z^2}
-\dfrac{2}{A^3}\left(\dfrac{\pa A}{\pa z}\right)^2\dfrac{\pa^2 A}{\pa \ov z^2}-\dfrac{8}{A^3}\dfrac{\pa A}{\pa z}\dfrac{\pa A}{\pa \ov z}\dfrac{\pa^2 A}{\pa z \pa\ov z}.
\end{multline}

Now the relation $g_{\ti B, D}(z_0)=1$ using \eqref{Fuks_tensor} and \eqref{der-A} gives
\[
\big|h_0(z_0)\big|^4 \left|\phi'(z_0)\right|^2 \left|\psi''(z_0)\right|^2=\big|h_0(z_0)\big|^4 \left|\phi'(z_0)\right|^4.
\]
Observe that $h_0(z_0)\neq 0$ and $\phi'(z_0)\neq 0$. Hence the above identity implies
\[
\left|\psi''(z_0)\right|^2=\left|\phi'(z_0)\right|^2.
\]
As a consequence, applying a unitary transformation to the basis $\{\psi\}$, we may assume that
\begin{equation}\label{relation_1}
\psi''(z_0)=\phi'(z_0).
\end{equation}

The relation $\tfrac{\pa g_{\ti B,D}}{\pa z}(z_0)=0$ using \eqref{relation_3} and \eqref{der-A} gives
\begin{multline*}
\big|h_0(z_0)\big|^2\,\left|\phi'(z_0)\right|^2\,\ov{\psi''}(z_0)\,\psi'''(z_0)+ \ov{h_0(z_0)}\,h_0'(z_0)\,\left|\phi'(z_0)\right|^2\,\left|\psi''(z_0)\right|^2\\
-2\big|h_0(z_0)\big|^2\,\ov{\phi'}(z_0)\,\phi''(z_0)\,\left|\psi''(z_0)\right|^2=0.
\end{multline*}
Using (\ref{relation_1}) in the above sum, we obtain the relation
\begin{equation}\label{relation_2}
h_0(z_0)\,\psi'''(z_0)+h_0'(z_0)\,\phi'(z_0)-2h_0(z_0)\,\phi''(z_0)=0.
\end{equation}

Finally, we compute the curvature. It follows from \eqref{2nd-der-g}, by a straightforward but lengthy calculation, that
\begin{align*}
A^4&\dfrac{-\pa^2 g_{\ti B,D}}{\pa z \pa\ov z}=-|h_0|^8\vert\phi'\vert^6\Big(|\psi'''|^2+\sum_{j=1}^{\infty} | h_j'''|^2\Big)-2|h_0|^6\,|\phi'|^6\,\text{Re}\left(h_0\ov{h_0'\psi''}\psi'''\right)\\
&-|h_0|^6\,\vert h_0'\vert^2\,\vert\phi'\vert^6\,\vert\psi''\vert^2-\vert h_0\vert^6\vert\phi'\vert^8\vert\psi''\vert^2+4\vert h_0\vert^8\vert\phi'\vert^4
\Re \left(\phi' \ov{\phi''}\ov{\psi''}\psi'''\right)\\
&+4\vert h_0\vert^6 \vert\phi'\vert^4\vert\psi'' \vert^2\Re \left(h_0\ov{h_0'\phi'}\phi''\right)+2\vert h_0\vert^8\vert\phi'\vert^4\vert\psi''\vert^4-4\vert h_0\vert^8\vert\phi'\vert^4\vert \phi''\vert^2\vert\psi''\vert^2
\end{align*}
at the point $z_0$. In the above equation and in the subsequent steps, if not mentioned, all the terms and partial derivatives are evaluated at the point $z_0$. Now making use of the relation (\ref{relation_1}), the above equation can be rewritten as
\begin{multline*}
\vert h_0\vert^8 \vert\phi'\vert^8\dfrac{-\pa^2 g_{\ti B,D}}{\pa z \pa\ov z}(z_0) = 2\vert h_0\vert^8\vert \phi' \vert^8-\vert h_0\vert^6 \vert \phi'\vert^{10}-\vert h_0\vert^8\vert \phi'\vert^6 \sum_{j=1}^{\infty} \vert h_j''' \vert^2 
-\vert h_0\vert^6\vert \phi' \vert^6 \Big\{ \vert h_0\vert^2 \vert \psi''' \vert^2 \\
+\vert h_0'\vert^2 \vert \phi' \vert^2 + 4 \vert h_0 \vert^2 \vert \phi'' \vert^2 + 2\Re \big(h_0\ov{h_0'\phi'}\psi''' \big)
-4 \vert h_0 \vert^2 \Re \big(\ov{\phi''}\psi''' \big)-4\Re \big(h_0\ov{h_0' \phi'}\phi'' \big)\Big\}.
\end{multline*}
Here one can see that, the terms inside the curly bracket above is exactly equal to
\[
\vert h_0 \psi''' +h_0' \phi' -2 h_0 \phi'' \vert^2,
\]
which vanishes by vitue of the relation (\ref{relation_2}). Hence we finally arrive at
\begin{equation}\label{Riem_curv}
R_{\ti B, D}(z_0)=-\dfrac{\pa^2 g_{\ti B,D}}{\pa z \pa\ov z}(z_0)=2-\dfrac{\big\vert \phi'(z_0) \big\vert^2}{\big\vert h_0(z_0) \big\vert^2}-\dfrac{1}{\big\vert \phi'(z_0) \big\vert^2}\sum_{j=1}^{\infty}\big\vert h_j'''(z_0)\big\vert^2.
\end{equation}
The right hand side of the above identity is finite thanks to the Cauchy estimates. The proposition now follows from Lemma~\ref{properties_I}. \qed
\end{proofof}

The following result is an immediate consequence of Proposition~\ref{Sec curv n=1}.
\begin{cor}
The Gaussian curvature of the Kobayashi--Fuks metric on  bounded planar domains is strictly bounded above by $2$.
\end{cor}

Now we localize the domain functions $I'_{D}$ and $I''_{D}$.

\begin{prop}\label{loc-I}
Let $D\subset \mathbb{C}$ be a bounded domain, $p^0 \in \pa D$ a local peak point, and $U$ a sufficiently small neighborhood of $p^0$. Then
\[
\lim\limits_{z\to p^0}\frac{I'_{D}(z)}{I'_{U \cap D}(z)}=\lim\limits_{z\to p^0}\frac{I''_{D}(z)}{I''_{U \cap D}(z)}=1.
\]
\end{prop}

\begin{proof}
We will present the proof only for $I''_{D}$ here. The proof for $I'_{D}$ follows in an exactly similar manner. Let $h$ be a local holomorphic peak function for $p^0$ defined on a neighborhood $U$ of $p^0$. Shrinking $U$ if necessary, we can assume that $h$ is nonvanishing on $U\cap \ov{D}$.  Now choose any neighborhood $V$ of $p^0$ such that  $V\subset\subset U$. Then there is a constant $b \in (0,1)$ such that $|h|\leq b$ on $\ov{(U\setminus V)\cap D}$. Let us choose a cut-off function $\chi\in C_0^{\infty}(U)$ satisfying $\chi =1$ on $V$ and $0\leq \chi \leq 1$ on $U$. Given any $\z \in V$, a function $f \in A^2(U\cap D)$, and an integer $k\geq 1$, set
\[
\phi(z)=8\log \vert z - \z\vert \quad \text{and} \quad  \al_k=\bar{\pa}(\chi f h^k).
\]
Then $\phi$ is a subharmonic function on $\mf{C}$ and $\al_k$ is a $\ov \pa$-closed, smooth $(0,1)$-form on $D$ with $\supp \al_k\subset (U\setminus V)\cap D$. Now as in \cite{Krantz-Yu}, applying Theorem~4.2 of \cite{Hor}, we get a solution $u$ to the equation $\bar{\pa}u=\al_k$ on $D$ such that
\begin{equation}
\frac{\pa^{|A|+|B|}u}{\pa z^A \pa\ov z^B}(\zeta)=0 \text{ for all multi-indices } A,B \text{ with } |A|+|B|\leq 3,
\end{equation}
and \begin{equation}\label{B1}
\|u\|_{D}\leq cb^k\|f\|_{U \cap D},
\end{equation}
where $c$ is a constant independent of $\z$ and $k$.

Now let $f \in A^2(U \cap D)$ be a maximizing function for $I''_{U \cap D}(\zeta)$, i.e., $\|f\|_{U \cap D}=1$, $f(\zeta)=f'(\zeta)=f''(\zeta)=0$, and $g^{-3}_{\ti B,U \cap D}(\zeta)|f'''(\zeta)|^2=I''_{U \cap D}(\zeta)$. Choose $u$ as above and set $F_k=\chi fh^k-u$. Then $F_k \in A^2(D)$ and it follows from (\ref{B1}) that
\begin{equation} \label{B2}
\|F_k\|_{D}\leq \|\chi fh^k\|_{D}+\|u\|_{D} 
\leq \|f\|_{U \cap D}+cb^k\|f\|_{U \cap D}
=(1+cb^k).
\end{equation}
Therefore, setting $f_k=F_k/\|F_k\|_D$, we see that $f_k\in A^2(D)$, $\|f_k\|_{D}=1$, and $f_k(\zeta)=f_k'(\zeta)=f_k''(\zeta)=0$. Moreover, by the maximality of $I''_{D}(\zeta)$, estimate (\ref{B2}), and (b) of Proposition \ref{Monotone}, one obtains
\begin{align*}
I''_{D}(\zeta)&\geq g^{-3}_{\ti B,D}(\zeta)|f_k'''(\zeta)|^2\\
&=\|F_k\|_D^{-2}|h^k(\zeta)|^2\, g^{-3}_{\ti B,D}(\zeta)|f'''(\zeta)|^2\\
&\geq \|F_k\|_D^{-2}|h(\zeta)|^{2k}\left(\frac{K_{D}(\zeta)}{K_{U \cap D}(\zeta)}\right)^{6} \left(\frac{J_{D}(\zeta)}{J_{U \cap D}(\zeta)}\right)^3\,g^{-3}_{\ti B,U \cap D}(\zeta)|f'''(\zeta)|^2\\
&\geq \frac{|h(\zeta)|^{2k}}{(1+cb^k)^2}\left(\frac{K_{D}(\zeta)}{K_{U \cap D}(\zeta)}\right)^{6} \left(\frac{J_{D}(\zeta)}{J_{U \cap D}(\zeta)}\right)^3\,I''_{U \cap D}(\zeta).
\end{align*}
This implies that
\[
\frac{I''_{D}(\zeta)}{I''_{U \cap D}(\zeta)}\geq \frac{|h(\zeta)|^{2k}}{(1+cb^k)^2}\left(\frac{K_{D}(\zeta)}{K_{U \cap D}(\zeta)}\right)^{6} \left(\frac{J_{D}(\zeta)}{J_{U \cap D}(\zeta)}\right)^3.
\]
Note that $h(p^0)=1$. By \eqref{K-loc-KY}, Proposition~2.1, and Proposition~2.4 of \cite{Krantz-Yu},
\[
\lim_{\zeta\to p^0}\frac{K_{D}(\zeta)}{K_{U \cap D}(\zeta)}=\lim_{\zeta\to p^0}\frac{J_{D}(\zeta)}{J_{U \cap D}(\zeta)}=1.
\]
Hence, letting $\zeta \to p^0$ in the above inequality, we get
\[
\liminf\limits_{\zeta \to p^0}\frac{I''_{D}(\zeta)}{I''_{U \cap D}(\zeta)}\geq (1+cb^k)^{-2}.
\]
Now letting $k\to \infty$, as $0\leq b<1$, and $c$ is independent of $k$, we obtain
\begin{equation}\label{geq}
\liminf\limits_{\zeta \to p^0}\frac{I''_{D}(\zeta)}{I''_{U \cap D}(\zeta)}\geq 1.
\end{equation} 

On the other hand, consider a candidate function $\eta$ for $I''_{D}(\zeta)$, i.e., $\eta \in A^2(D)$, $\|\eta\|_{D}=1$, and $\eta(\zeta)=\eta'(\zeta)=\eta''(\zeta)=0$. Now for $z\in U \cap D$, let's define 
\[
\ga(z)=\dfrac{\eta(z)}{\|\eta\|_{U \cap D}}.
\]
Then clearly $\ga\in A^2(U \cap D)$ with $\|\ga\|_{U \cap D}=1$, and $\ga(\zeta)=\ga'(\zeta)=\ga''(\zeta)=0$. Therefore the maximality of $I''_{U \cap D}(\zeta)$ implies
\begin{multline*}
I''_{U \cap D}(\zeta)\geq g^{-3}_{\ti B, U \cap D}(\zeta)|\ga'''(\zeta)|^2 =\|\eta\|^{-2}_{U \cap D}\left(\dfrac{g_{\ti B, U \cap D}(\zeta)}{g_{\ti B, D}(\zeta)}\right)^{-3} g^{-3}_{\ti B, D}(\zeta)|\eta'''(\zeta)|^2\\
\geq \left(\dfrac{g_{\ti B, U \cap D}(\zeta)}{g_{\ti B, D}(\zeta)}\right)^{-3}I''_{D}(\zeta).
\end{multline*}
The last inequality above follows from $\|\eta\|_{U \cap D}\leq 1$ and the fact that $\eta$\, is an arbitrary candidate function for $I''_{D}(\zeta)$. Thus we obtain
\begin{equation}\label{inequality-2}
\frac{I''_{D}(\zeta)}{I''_{U \cap D}(\zeta)}\leq \left(\dfrac{g_{\ti B, U \cap D}(\zeta)}{g_{\ti B, D}(\zeta)}\right)^{3}.
\end{equation}
By (II) of Theorem~\ref{loc}, the right hand side converges to $1$ as $\z \to p^0$, and therefore,
\begin{align}\label{leq}
\limsup_{\zeta \to p^0}\frac{I''_{D}(\zeta)}{I''_{U \cap D}(\zeta)}\leq 1.
\end{align}
From \eqref{geq} and \eqref{leq}, we conclude that
\begin{equation*}
\lim\limits_{\zeta\to p^0} \frac{I''_{D}(\zeta)}{I''_{U \cap D}(\zeta)}=1,
\end{equation*}
as required. \qed
\end{proof}

\begin{lem}\label{seq}
Suppose $\{a_j\},\{b_j\},\{c_j\}$ and $\{d_j\}$ are real sequences with $b_j,d_j> 0$ and
\[
\lim_{j\to \infty}\dfrac{a_j}{b_j}=1\quad \text{and} \quad \lim_{j\to \infty}\dfrac{c_j}{d_j}=1.
\]
Then we have 
\[
\lim_{j\to \infty}\dfrac{a_j+c_j}{b_j+d_j}=1.
\]
\end{lem}

\begin{proof}
To verify this claim let $\epsilon>0$. Then there exists $N\in \mbb{N}$ such that
\[
\bigg|\dfrac{a_j}{b_j}-1\bigg|<\dfrac{\epsilon}{2}\quad \text{and}\quad \bigg|\dfrac{c_j}{d_j}-1\bigg|<\dfrac{\epsilon}{2}\quad \text{for}\,\,j\geq N.
\]
This, in particular implies that
\[
\dfrac{|a_j-b_j|}{b_j+d_j}<\dfrac{\epsilon}{2}\quad \text{and}\quad \dfrac{|c_j-d_j|}{b_j+d_j}<\dfrac{\epsilon}{2}\quad \text{for}\,\,j\geq N.
\]
Hence we have
\begin{align*}
\bigg|\dfrac{a_j+c_j}{b_j+d_j}-1\bigg|\leq \dfrac{|a_j-b_j|}{b_j+d_j}+\dfrac{|c_j-d_j|}{b_j+d_j}<\epsilon\quad \text{for}\,\,j\geq N,
\end{align*}
proving the lemma. \qed
\end{proof}

We are now in a state to give a proof of Theorem~\ref{loc}.

\begin{proofof}{Theorem~\ref{loc}~(III)}
By Proposition \ref{Sec curv n=1}, let us write
\[
R_{\ti B, D}(z)=2-\ti E_{D}(z)-\ti F_{D}(z),
\]
where 
\[
\ti E_{D}(z)=\dfrac{I'_{D}(z)}{K_{D}(z)}\quad \text{and}\quad \ti F_{D}(z)=\dfrac{I''_{D}(z)}{I'_{D}(z)}.
\]
From Proposition~\ref{loc-I} and \eqref{K-loc-KY}, we have
\[
\lim\limits_{z\to p^0}\frac{\ti E_{D}(z)}{\ti E_{U \cap D}(z)}=1\quad \text{and}\quad \lim\limits_{z\to p^0}\frac{\ti F_{D}(z)}{\ti F_{U \cap D}(z)}=1,
\]
for a small enough neighborhood $U$ of $p^0$. Now (III) follows immediately from Lemma~\ref{seq}. \qed
\end{proofof}

\section{Boundary behavior}
The proof of Theorem~\ref{bdy-n} is based on Pinchuk's scaling method and we begin by recalling the change of coordinates associated to this method. We will use the standard notation $z=('z,z_n)$, where $'z=(z_1,\ldots,z_{n-1})$. Throughout this section $D$ is a $C^2$-smoothly bounded strongly pseudoconvex domain in $\mbb C^n$ and $\rho$ is a $C^2$-smooth local defining function for $D$ defined on a neighborhood $U$ of a point $p^0\in \pa D$. Without loss of generality, we assume that
\begin{equation}\label{normal-Re-z_n}
\nabla_{\ov z} \rho(p^0)=({'0},1) \text{ and } \frac{\pa \rho}{\pa z_n}(z)\neq 0 \text{ for all } z\in U.
\end{equation}
Here, $\nabla_z \rho=(\pa \rho/\pa z_1, \ldots, \pa \rho/\pa z_n)$ and we write $\nabla_{\ov z}\rho=\ov{\nabla_z \rho}$. Note that the gradient $\nabla \rho=2\nabla_{\ov z}\rho$.

\subsection{Change of coordinates}

The following lemma from \cite{Pin} illustrates the change of coordinates near strongly pseudoconvex boundary points.

\begin{lem}\label{C0}
There exist a family of biholomorphic mappings $h^{\zeta}:\mbb{C}^n\to \mbb{C}^n$ depending continuously on $\zeta \in \pa D \cap U$, satisfying the following conditions:
\begin{itemize}
\item[\em{(a)}] $h^{p^0}=\mathbb{I}$.
\item[\em{(b)}] $h^{\zeta}(\zeta)=0$.
\item[\em{(c)}] The local defining function $\rho^{\zeta}=\rho\,\circ (h^{\zeta})^{-1}$ of the domain $D^{\zeta}=h^{\zeta}(D)$ near the origin has the form 
\[
\rho^{\zeta}(z)=2\text{Re}\big(z_n+K^{\zeta}(z)\big)+H^{\zeta}(z)+o(|z^2|)
\]
in a neighborhood of the origin, where
\[
K^{\zeta}(z)=\sum\limits_{\mu,\nu=1}^n a_{\mu \nu}(\zeta)z_{\mu}z_{\nu}\quad \text{and} \quad H^{\zeta}(z)=\sum\limits_{\mu,\nu=1}^n a_{\mu \ov{\nu}}(\zeta)z_{\mu}\ov{z}_{\nu}
\]
with $K^{\zeta}('z,0)\equiv 0$ and $H^{\zeta}('z,0)\equiv |'z|^2$.
\item[\em{(d)}] The biholomorphism $h^{\zeta}$ takes the real normal $\eta_{\zeta}=\{z=\z+2t \nabla_{\ov z}\rho(\z): t \in \mathbb{R}\}$ to $\pa D$ at $\z$ into the real normal $\{'z=y_n=0\}$ to $\pa D^{\zeta}$ at the origin.
\end{itemize}
\end{lem}

The definition of the map $h^{\z}$ and its derivative will play an important role in the computation of the boundary asymptotics and so we quickly recall its construction. We fix $\z \in \pa D \cap U$. The map $h^{\zeta}$ is a polynomial automorphism of $\mathbb{C}^n$ defined as the composition $h^{\zeta}(z)=\phi_3^{\zeta}\circ \phi_2^{\zeta}\circ \phi_1^{\zeta}(z)$, where the maps $\phi_i^{\zeta}:\mathbb{C}^n \to \mathbb{C}^n$ are biholomorphisms defined as follows: The map $w=\phi_1^{\z}(z)$ is an affine transformation given by
\begin{equation}\label{phi_1-defn}
\begin{aligned}
w_j & =\frac{\pa \rho}{\pa \ov z_n}(\zeta)(z_j-\zeta_j)-\frac{\pa \rho}{\pa \ov z_j}(\zeta)(z_n-\zeta_n)\quad \text{for}\quad j=1,\ldots,n-1,\\
w_n & =\sum\limits_{\nu=1}^n \frac{\pa \rho}{\pa z_{\nu}}(\zeta)(z_{\nu}-\zeta_{\nu}).
\end{aligned}
\end{equation}
The map $\phi_1^{\zeta}$ is nonsingular by \eqref{normal-Re-z_n} and it takes the point $\zeta$ to the origin. We relabel the new coordinates $w$ as $z$. Then the Taylor series expansion of the local defining function $\rho \circ (\phi_1^{\z})^{-1}$ for the domain $\phi_1^{\z}(D)$ near the origin has the form
\begin{equation}\label{T-exp}
2\text{Re}\left(z_n+\sum\limits_{\mu,\nu=1}^n a_{\mu\nu}(\zeta)z_{\mu}z_{\nu}\right)+H^{\z}(z)+o(\vert z \vert^2),
\end{equation}
where $H^{\z}(z)$ is a Hermitian form.

The map $w=\phi_2^{\zeta}(z)$ is given by 
\begin{equation}\label{phi_2-defn}
w=\left('z, z_n+\sum\limits_{\mu,\nu=1}^{n-1}a_{\mu\nu}(\zeta)z_{\mu}z_{\nu}\right)
\end{equation}
and is a polynomial automorphism. Relabelling the new coordinates $w$ as $z$, the Taylor series expansion of the local defining function $ \rho \circ (\phi^{\z}_1)^{-1} \circ (\phi^{\z}_2)^{-1}$ for the domain $\phi_2^{\z} \circ \phi_1^{\z}(D)$ has the form \eqref{T-exp} with $
a_{\mu\nu}=0$ for $1\leq \mu, \nu \leq n-1$.

Finally, the map $\phi^{\z}_3$ is chosen so that the Hermitian form $H^{\z}(z)$ satisfies $H^{\z}('z,0)=\vert 'z\vert^2$. Since $D$ is strongly pseudoconvex and the complex tangent space to $\pa D$ at $\z$ is given by $z_n=0$ in the current coordinates, the form $H^{\zeta}('z,0)$ is strictly positive definite. Hence there exists a unitary map $U^{\zeta}:\mbb C^{n-1}\rightarrow \mbb C^{n-1}$ such that $H^{\zeta}(U^{\zeta}('z),0)$ is diagonal with diagonal entries $\la_1^{\zeta},\ldots,\la_{n-1}^{\zeta}>0$. Now consider the stretching map $L^{\zeta}=\text{diag}\{(\la_1^{\zeta})^{-1/2},\ldots,(\la_{n-1}^{\zeta})^{-1/2}\}$. Then the linear map $A^{\zeta}:\mbb C^{n-1}\to \mbb C^{n-1}$ given by $A^{\zeta}:=L^{\zeta}\circ U^{\zeta}$ satisfies $H^{\zeta}(A^{\zeta}('z),0)=|'z|^2$. Note that $U^{\z}$ and $L^{\z}$, and hence $A^{\z}$ can be chosen to depend continuously on $\z$. Thus, if we define $w=\phi^{\z}_3(z)$ by
\[
w=(A^{\zeta}('z),z_n),
\]
and relabel $w$ as $z$, then the local defining function $\rho^{\z}=\rho\circ (h^{\z})^{-1}$ for the domain $D^{\z}=h^{\z}(D)$ near the origin has the Taylor series expansion as in (c). Also, it is evident from the construction of the maps $h^{\z}$ that they satisfy (a), (b), (d), and that they depend continuously on $\z$.

\subsection{Scaling of $D$}\label{Dialation}
By strong pseudoconvexity, shrinking $U$ if necessary, there exist local holomorphic coordinates $z_1,\ldots, z_n$ on $U$ in which $p^0=0$, and
\begin{equation}\label{normal form}
\rho(z)=2\Re z_n+|'z|^2+o\big(\Im z_n,|'z|^2\big), \quad z \in U,
\end{equation}
and a constant $0<r<1$ such that
\begin{align}\label{subset1}
U \cap D\subset \Om:=\big\{z \in \mbb C^n: 2\Re z_n+r|'z|^2<0\big\}.
\end{align}
Henceforth, we will be working in the above coordinates, and with $U,\rho$ and $p^0$ as above.

Let us consider a sequence of points $p^j$ in $D$ that converges to $p^0=0$ on $\pa D$. For $j$ sufficiently large, and without loss of generality we assume that for all $j$, $p^j \in U$ and there exists a unique $\z^j\in \pa D \cap U$ that is closest to $p^j$. Define $\de_j:= \text{d}(p^j,\pa D)=|p^j-\z^j|$. Note that $\z^j\to p^0=0$ and $\de_j\to 0$ as $j\to \infty$. For each $\z^j$, denote by $h_j$ the map $h^{\z^j}$ given by Lemma \ref{C0}. Denoting $\phi_i^{\z^j}$ by $\phi_i^j$ for $j=1,2,3$, we have $h_j=\phi_1^j\circ \phi_2^j\circ \phi_3^j$. Also set $\rho_j=\rho^{\z^j}$. Then by Lemma \ref{C0}, near the origin,
\begin{align*}
\rho_j(z)= 2\text{Re}\,\big(z_n+K_j(z)\big)+H_j(z)+o\big(|z|^2\big),
\end{align*}
where $K_j=K^{\z^j}$ and $H_j=H^{\z^j}$. Moreover, thanks to the strong pseudoconvexity of 
$\pa D$ near $p^0=0$, shrinking $U$ if necessary and taking a smaller $r$ in \eqref{subset1}, we have
\begin{align}\label{subset2}
h_j(U\cap D) \subset \Om
\end{align}
for all large $j$. Note that by Theorem~\ref{loc}, it is enough to prove Theorem~\ref{bdy-n} for the domain $U\cap D$, shrinking $U$ if necessary. Set $D_j=h_j(U\cap D)$, $q^j=h_j(p^j)$, and $\eta_j=\text{d}(q^j,\pa D_j)$.

Now consider the anisotropic dilation map $\La_j : \mathbb{C}^n \to \mathbb{C}^n$ defined by
\begin{equation}\label{defn-Lambda_j}
\La_j(z)=\left(\frac{z_1}{\sqrt{\eta_j}},\ldots,\frac{z_{n-1}}{\sqrt{\eta_j}},\frac{z_n}{\eta_j}\right).
\end{equation}
Set $\ti D_j=\La_j(D_j)=\La_j\circ h_j(U\cap D)$. We will call the maps $S_j:=\La_j\circ h_j$ the scaling maps and the domains $\widetilde{D}_j=S_j(U\cap D)$ the scaled domains. Note that since $S_j(p^j)=\La_j\circ h_j(p^j)=('0,-1)$, each $\ti D_j$ contains the point $('0,-1)$ and we will denote this point by $b^*$.  A defining function for $\ti D_j$ near the origin, is given by
\[
\ti \rho_j(z)=\frac{1}{\eta_j} \rho_j\big(\La_j^{-1}(z)\big)=2\text{Re}\Big(z_n+\frac{1}{\eta_j}K_j\big(\La_j^{-1}(z)\big)\Big)+\frac{1}{\eta_j}H_j\big(\La_j^{-1}(z)\big)+o\big(\eta_j^{1/2}|z|^2\big).
\]
Since $K_j(z)$ and $H_j(z)$ satisfy condition (c) of Lemma~\ref{C0}, it follows that
\[
\lim_{j\to \infty}\frac{1}{\eta_j}K_j(\La_j^{-1}z)=0\quad \text{and}\quad \lim_{j\to \infty}\frac{1}{\eta_j}H_j(\La_j^{-1}z)=|'z|^2
\]
in $C^2$-topology on compact subsets of $\mathbb{C}^n$. Evidently, $o(\eta_j^{1/2}|z|^2)\to 0$ as $j \to \infty$ in $C^2$-topology on any compact set of $\mbb C^n$. Thus, the defining functions $\ti\rho_j$ converge in $C^2$ topology on compact subsets of $\mbb C^n$ to 
\[
\rho_{\infty}(z)=2\text{Re}\,z_n+|'z|^2.
\]
Hence our scaled domains $\ti D_j$ converge in the local Hausdorff sense to the Siegel upper half-space
\[
D_{\infty}=\big\{z\in \mbb C^n:2\text{Re}\,z_n+|'z|^2<0\big\}.
\]

\subsection{Stability of the Kobayashi--Fuks metric}

\begin{prop}\label{stability-Fuks}
For $z \in D_{\infty}$ and $u\in \mathbb{C}^n \setminus \{0\}$, we have
\[
g_{\ti B, \ti D_j}(z)\to g_{\ti B, D_{\infty}}(z),  \quad \tau_{\ti B, \ti D_j}(z,u)\to \tau_{\ti B, D_{\infty}}(z,u) \quad \text{and}\quad \Ric_{\ti B, \ti D_j} (z,u) \to \Ric_{\ti B, D_{\infty}}(z,u)
\]
as $j \to \infty$. Moreover, the first convergence is uniform on compact subsets of \,$D_{\infty}$ and the second and third convergences are uniform on compact subsets of\, $D_{\infty}\times \mbb C^n$.
\end{prop}

\begin{proof}
Since the Kobayashi--Fuks metric on the domain $D$ has K\"ahler potential $\log(K^{n+1}_{D}g_{B,D})$, i.e.,
\[
g^{\ti B,D}_{\al \ov \be} = \frac{\pa^2 \log(K^{n+1}_{D}g_{B,D})}{\pa z_{\al}\pa \ov z_{\be}},
\]
all that is required is to show that
\[
K^{n+1}_{\ti D_j}g_{B,\ti D_j} \to K^{n+1}_{D_{\infty}}g_{B,D_\infty}
\]
uniformly on compact subsets of $D_{\infty}$, together with all derivatives. But this is an immediate consequence of the fact that $K_{\ti D_j} \to K_{D_{\infty}}$ together will all derivatives on compact subsets of $D_{\infty}$. This fact can be established from a Ramadanov type result \cite[Lemma~2.1]{BBMV1}, and for the details we refer the reader to Lemma~5.3 of \cite{BV-ns} with the note that by taking $d=0$ there, $K_{\ti D_j}=K_{\ti D_j,0}$ and $K_{D_{\infty}}=K_{D_{\infty},0}$. \qed
\end{proof}

\subsection{Boundary asymptotics}
Recall that $S_j=\La_j \circ h_j$, $S_j(U\cap D)=\ti D_j$, and $S_j(p^j)=b^{*}=('0,-1)$. Denoting the matrix of a linear map by itself, we have
\begin{equation}\label{der-S_j}
S_j'(p^j)=\La_j h_j'(p^j)=
\La_j \cdot \phi_3^j \cdot (\phi_2^j)'\big(\phi_1^j(p^j)\big) \cdot (\phi_1^j)'(p^j).
\end{equation}
Note that from the definition of $\phi_1^j$,
\begin{equation}\label{phi_1-matrix}
(\phi_{1}^j)'(p^j)= 
\begin{pmatrix}
\dfrac{\pa \rho}{\pa \ov z_n}(\z^j) & 0 & \cdots & 0 & -\dfrac{\pa\rho}{\pa \ov z_1}(\z^j)\\
0 & \dfrac{\pa \rho}{\pa \ov z_n}(\z^j) & \cdots & 0 & -\dfrac{\pa\rho}{\pa \ov z_2}(\z^j)\\
\vdots & \vdots & \ddots & \vdots & \vdots\\
0 & 0 & \cdots & \dfrac{\pa \rho}{\pa \ov z_n}(\z^j) & -\dfrac{\pa\rho}{\pa \ov z_{n-1}}(\z^j)\\
\dfrac{\pa\rho}{\pa z_1}(\z^j) & \dfrac{\pa\rho}{\pa z_2}(\z^j) & \cdots & \dfrac{\pa\rho}{\pa z_{n-1}}(\z^j) & \dfrac{\pa\rho}{\pa z_n}(\z^j)
\end{pmatrix}.
\end{equation}
Also, since 
\begin{equation}\label{p^j-form}
p^j=\z^j-\de_j \frac{\nabla_{\ov z} \rho(\z^j)}{\big \vert \nabla_{\ov z} \rho(\z^j) \big \vert},
\end{equation}
we have
\begin{equation}\label{phi_1(p^j)-form}
\phi_1^j(p^j)=\Big({}'0,-\de_j\big|\nabla_{\ov z}\rho(\z^j)\big|\Big).
\end{equation}
Therefore, from the definition of $\phi_2^j$, we have
\begin{align*}
(\phi_2^j)'\big(\phi_1^j(p^j)\big)=\mathbb{I}_n.
\end{align*}
Finally, recall that $\phi^j_3(z)=(A^j('z),z_n)$ where $A^j:=A^{\z^j}:\mathbb{C}^{n-1} \to \mathbb{C}^{n-1}$ are linear maps satisfying $A^j \to \mathbb{I}_{n-1}$. Therefore,
\[
\phi^j_3=\begin{bmatrix} A^j_{p,q} & 0\\ 0 &1\end{bmatrix},
\]
where $A^j_{p,q}$ are the entries of the matrix of $A^j$. Thus,
\begin{equation}\label{lim-Dh_j}
h_j'(p^j) = \begin{pmatrix}
A^j_{1,1}\dfrac{\pa\rho}{\pa \ov z_n}(\z^j)  & \cdots & A^j_{1, n-1}\dfrac{\pa\rho}{\pa \ov z_n}(\z^j) & -\sum\limits_{\nu=1}^{n-1}
A^j_{1, \nu}\dfrac{\pa\rho}{\pa \ov z_\nu}(\z^j)\\
\vdots & \cdots & \vdots & \vdots \\
A^j_{n-1, 1}\dfrac{\pa\rho}{\pa \ov z_n}(\z^j) & \cdots & A^j_{n-1, n-1}\dfrac{\pa\rho}{\pa \ov z_n}(\z^j) & -\sum\limits_{\nu=1}^{n-1}A^j_{n-1, \nu}\dfrac{\pa\rho}{\pa \ov z_\nu}(\z^j)\\
\dfrac{\pa\rho}{\pa z_1}(\z^j) & \cdots & \dfrac{\pa\rho}{\pa z_{n-1}}(\z^j) & \dfrac{\pa\rho}{\pa z_n}(\z^j)
\end{pmatrix} \to \mathbb{I}_{n}
\end{equation}
in the operator norm.

We also note that as $\phi_2^j$ and $\phi_3^j$ fix points on the $\Re z_n$-axis, we have from \eqref{phi_1(p^j)-form},
\[
q^j=h_j(p^j)=\Big({}'0,-\de_j\big|\nabla_{\ov z}\rho(\z^j)\big|\Big).
\]
As the normal to $\pa D_j$ at $0$ is the $\Re z_n$-axis and $\eta_j=\text{d}(q^j,\pa D_j)$, we have $\eta_j=\de_j\,|\nabla_{\ov z}\rho(\z^j)|$ and hence
\begin{equation}\label{lim-de_j/eta_j}
\lim_{j \to \infty} \frac{\eta_j}{\delta_j} =1.
\end{equation}

Now consider the Cayley transform $\Phi$ defined by
\begin{equation}\label{C-T}
\Phi(z_1,\ldots,z_n)=\left(\dfrac{\sqrt 2\,{}'z}{z_n-1},\dfrac{z_n+1}{z_n-1}\right).
\end{equation}
It can be shown that $\Phi$ maps $D_{\infty}$ biholomorphically onto $\mbb{B}^n$. We also note that $b^{*}=('0,-1) \in D_{\infty}$, $\Phi(b^{*})=0$, and
\begin{equation}\label{der-Phi}
\Phi'(b^*)=-\diag\{1/\sqrt 2,\ldots,1/\sqrt 2, 1/2\}.
\end{equation}

We now present the proof of Theorem~\ref{bdy-n}.

\begin{proofof}{Theorem~\ref{bdy-n}}
Note that by the localization result Theorem~\ref{loc}, it is enough to compute the asymptotics for the domain $U \cap D$.

(I) By invariance of the Kobayashi--Fuks metric,
\begin{align*}
\tau_{\ti B, U \cap D}(p^j,u)=\tau_{\ti B, \ti D_j}\big(b^*,S_j'(p^j)u\big).
\end{align*}
Note that
\[
S_j'(p^j)u=\La_j \big(h_j'(p^j)u\big)=\Big(\eta_j^{-1/2}\,\, {'\big(h_j'(p^j)u\big)}, \eta_j^{-1}\big(h_j'(p^j)u\big)_n\Big),
\]
and so by \eqref{lim-Dh_j},
\[
\eta_j S_j'(p^j)u \to ('0, u_n)
\]
uniformly in unit vectors $u$. Therefore, by Proposition~\ref{stability-Fuks},
\begin{align*}
\lim _{j \to\infty} \de_j\,\tau_{\ti B, U \cap D}(p^j,u)=\lim_{j \to\infty} \frac{\de_j}{\eta_j} \tau_{\ti B, \ti D_j}\big(b^*,\eta_j S_j'(p^j)u\big)=\tau_{\ti B,  D_{\infty}}\big(b^*,('0,u_n)\big)\end{align*}
uniformly in unit vectors $u$. Now, all that is required is to compute the right hand side using the Cayley transform $\Phi$ from \eqref{C-T}, its derivative from \eqref{der-Phi}, and the transformation rule. Thus,
\[
\tau_{\ti B,  D_{\infty}}\big(b^*,('0,u_n)\big)=\tau_{\ti B, \mbb B^n}\bigg(0,\left('0,-\dfrac{u_n}{2}\right)\bigg)=\dfrac{1}{2}\sqrt{(n+1)(n+2)}|u_n|,
\]
where the last equality follows from Proposition~\ref{ball}, and this proves (I).

(II) For brevity, we write $u^j=u_H(p^j)$ and $u^0 =u_H(p^0)$. By invariance of the Kobayashi--Fuks metric,
\begin{align*}
\tau_{\ti B, U \cap D}\big(p^j,u_H(p^j)\big)=\tau_{\ti B, \ti D_j}\big(b^*, S_j'(p^j)u^j\big).
\end{align*}
Note that, since $u^j\in H_{q^j}(\pa D)$, we have from \eqref{lim-Dh_j}
\[
h_j'(p^j)u^j =(v_1^j,\ldots,v_{n-1}^j,0),
\]
where 
\[
v^j_l=\sum_{\nu=1}^{n-1} A^j_{l, \nu} \left(u^j_{\nu} \frac{\pa \rho}{\pa \ov z_n}(\z^j) - u^j_n \frac{\pa \rho}{\pa \ov z_{\nu}}(\z^j)\right),
\quad l=1,\ldots,n-1.
\]
Therefore,
\begin{align*}
S_j'(p^j)u^j =\La_j\big(h_j'(p^j)u^j\big)= \left(\frac{v^j_1}{\sqrt{\eta_j}}, \ldots, \frac{v^j_{n-1}}{\sqrt{\eta_j}},0\right).
\end{align*}
Observe that $v^j_l \to u^0_l$ and the convergence is uniform on unit vectors $u$ and so
\[
\sqrt{\eta_j} S_j'(p^j)u^j \to (u^0_1, \ldots, u^0_{n-1},0)
\]
and the convergence is uniform in unit vectors $u$. Hence, by Proposition~\ref{stability-Fuks},
\[
\sqrt{\de_j} \tau_{\ti B, U \cap D}\big(p^j,u_H(p^j)\big)  =\sqrt{\frac{\de_j}{\eta_j}}  \tau_{\ti B, \ti D_j}\big(b^*, \sqrt{\eta_j} S_j'(p^j)u^j\big) \to \tau_{\ti B, D_{\infty}}\big(b^{*}, ('u^0, 0)\big)
\]
uniformly in unit vectors $u$. Again, using the Cayley transform $\Phi$ from \eqref{C-T} and its derivative from \eqref{der-Phi}, the transformation rule gives
\begin{align*}
 \tau_{\ti B, D_{\infty}}\big(b^{*}, ('u^0, 0)\big)
=\tau_{\ti B, \mbb B^n}\bigg(0,\bigg(-\dfrac{'u^0}{\sqrt 2},0\bigg)\bigg)=\sqrt{\dfrac{1}{2}(n+1)(n+2) \vert 'u^0\vert^2},
\end{align*}
by Proposition~\ref{ball}. This proves (II) once we observe from \eqref{normal form} that $\mathcal{L}_{\rho}\big(p^0,u_H(p^0)\big)=\vert 'u^0\vert^2$.

(III) By the transformation rule for the Kobayashi--Fuks metric, we have

\begin{align}\label{transformation4}
g_{\ti B, U \cap D}(p^j)=g_{\ti B, \ti D_j}(b^*) \big\vert \det S'_j(p^j)\big\vert^2.
\end{align}
Note that
\[
\det S'_j(p^j)= \det \La_j \det h_j'(p^j) = \eta_j^{-(n+1)/2}\det h_j'(p^j),
\]
and so by \eqref{lim-Dh_j},
\[
\eta_j^{n+1} \vert \det S'_j(p^j)\vert^2 \to 1.
\]
Therefore,
\[
\de_j^{n+1} g_{\ti B, U \cap D}(p^j) = \left(\frac{\de_j}{\eta_j}\right)^{n+1} g_{\ti B, \ti D_j}(b^*) \eta_j^{n+1}\big\vert \det S'_j(p^j)\big\vert^2 \to g_{\ti B, D_{\infty}}(b^*).
\]
As before, using the Cayley transform $\Phi$ from \eqref{C-T} and its derivative from \eqref{der-Phi}, we obtain from the transformation rule,
\[
g_{\ti B,  D_\infty}(b^*)=g_{\ti B, \mbb B^n}(0)\,|\text{det}\,\Phi'(b^*)|^2=\dfrac{(n+1)^n(n+2)^n}{2^{n+1}},
\]
by Proposition~\ref{ball}. This completes the proof of (III).

(IV) By invariance and Proposition~\ref{stability-Fuks},
\[
R_{\ti B, U \cap D}(p^j)= R_{ \ti B, \ti D_j} (b^*) \to R_{\ti B, D_{\infty}}(b^*)
= R_{\ti B, \De}(0).
\]
By Proposition~\ref{ball}, we have for $z \in \De$,
\[
g_{\ti B,\De}(z)=\frac{6}{(1-|z|^2)^2},
\]
and therefore,
\[
R_{\ti B, \De}(z)=-\frac{1}{g_{\ti B,\De}(z)}\frac{\pa^2 \log g_{\ti B, \De}}{\pa z \pa \ov z}(z)=-\frac{1}{3},
\]
which completes the proof of (IV) and the theorem. \qed
\end{proofof}

\begin{rem}
The proof of Theorem~\ref{bdy-n} \em{(IV)} also follows from the observation that we can choose a sufficiently small disc $U$ centered at $p^0$ such that $U \cap D$ is simply connected thanks to the smoothness of $\pa D$. Therefore, $U \cap D$ is biholomorphic to the unit disc $\De$ and hence the Kobayashi--Fuks metric of $U \cap D$ has the constant Gaussian curvature $-1/3$ which follows from Proposition~\ref{ball}. Now, applying Theorem~\ref{loc}, we immediately obtain $R_{\ti B,  D}(p^j) \to -1/3$.

\end{rem}

\section{Existence of closed geodesics with prescribed homotopy class}

The proof of Theorem~\ref{geodesic} is based on the following result:
\begin{thm}[Herbort, {\cite[Theorem1.1]{Herbort1983}}]\label{Herbort}
Let $G \subset \mathbb{R}^N$ be a bounded domain such that $\pi_1(G)$ is nontrivial and the following conditions are satisfied:
\begin{itemize}
\item[\em{(i)}] For each $p\in \ov G$ there is an open neighborhood $U\subset \mbb R^n$, such that the set $G\cap U$ is simply connected.
\item[\em{(ii)}] The domain $G$ is equipped with a complete Riemannian metric $g$ which possesses the following property:
\begin{itemize}
\item[\em{(P)}] For each $S>0$ there is a $\de>0$ such that for every point $p\in G$ with $\text{d}(p,\pa G)<\de$ and every $X\in \mbb R^n$, $g(p,X)\geq S\|X\|^2$.
\end{itemize}
\end{itemize}
Then every nontrivial homotopy class in $\pi_1(G)$ contains a closed geodesic for $g$.
\end{thm}

\begin{proofof}{Theorem~\ref{geodesic}}
We will show that both the conditions in Theorem~\ref{Herbort} hold for $G=D$ and $g=ds^2_{\ti B, D}$. By the smoothness of $\pa D$, it is evident that condition (i) is satisfied. For condition (ii), note that from \eqref{B-KF} and the fact that $\text{Ric}_{B,D}(z,u)$ approaches $-1$ near the boundary of a strongly pseudoconvex domain (see \cite[Corollary~2]{Krantz-Yu}), there exists $C = C(D) > 0$ such that
\begin{align}\label{B-KF inequ}
\tau_{\ti B, D}(z,u)\geq C\tau_{B, D}(z,u)
\end{align}
for $z$ near the boundary of $D$ and unit vectors $u$. As both the Bergman and Kobayashi--Fuks metric are K\"ahler, this relation also holds for $z$ on any compact subset of $D$ and unit vectors $u$. Thus \eqref{B-KF inequ} holds for all $z \in D$ and $u \in \mathbb{C}^n$. This has the following two consequences. Firstly, since the Bergman metric dominates the Carath\'eodory metric on bounded domains (see Hahn \cite{Hahn}) and the Carath\'eodory metric is complete on smoothly bounded strongly pseudoconvex domains (see \cite{Jarn-Pflug-2013}, p.539), (\ref{B-KF inequ}) implies that the Kobayashi--Fuks metric on $D$ is complete. Secondly, as the Bergman metric on $D$ satisfies property (P) which was observed in the proof of Theorem~1.2 in \cite{Herbort1983}, (\ref{B-KF inequ}) also implies that the Kobayashi--Fuks metric on $D$ satisfies property (P) as well, and hence condition (ii) holds. This completes the proof of the theorem. \qed
\end{proofof}

\section{Some questions}
We conclude this article with the following questions:
\begin{itemize}
\item [(i)] Does the localization of holomorphic sectional curvature and Ricci curvature of the Kobayashi--Fuks metric near holomorphic peak points hold in dimensions $n\geq 2$?

\item [(ii)] Herbort studied the existence of geodesic spirals for the Bergman metric on strongly pseudoconvex domains \cite[Theorem~3.2]{Herbort1983}. Does the analog of this result hold for the Kobayashi--Fuks metric?

\item [(iii)] Is there an analog of the Donnelly-Fefferman's result on the $L^2$-cohomology of the Bergman metric \cite[Theorem~1.1]{Don-Fef} for the Kobayashi--Fuks metric? See also \cite{Don1, Don2} for simpler proofs based on the fact that the Bergman metric is given by a global potential. Note that the Kobayashi--Fuks metric is also given by a global potential.
\end{itemize}

\section*{References}
\bibliography{kfbib}

\end{document}